\documentclass[]{ensmath}

\EnsMath pages 1-7

\usepackage{epsfig}

\title[ON THE INCENTERS OF TRIANGULAR ORBITS IN ELLIPTIC BILLIARD]{ON THE INCENTERS OF TRIANGULAR\\ORBITS IN ELLIPTIC BILLIARD }

\author[O. ROMASKEVICH]{Olga \sn{Romaskevich}\thanks{
Supported in part by RFBR grants 
$12-01-31241$ mol-a and $12-01-33020$ mol-a-ved.} }

\begin{document}

\maketitle

\begin{abstract}
We consider $3$-periodic orbits in an elliptic billiard. Numerical experiments conducted by Dan Reznik have shown that the locus of the centers of inscribed circles of the corresponding triangles is an ellipse. We prove this fact by the complexification of the problem coupled with the complex law of reflection.
\end{abstract}

\section{The statement of the theorem and the idea of the proof}

Elliptic billiards are at the same time classical and popular subject (see, for example \cite{KT}, \cite{CM}, \cite{T1} and \cite{T2}) since they continue to deliver interesting problems. We will consider an ellipse and a billiard in it with the standard reflection law: the angle of incidence equals the angle of reflection. Let the trajectory from a point on the boundary repeat itself after two reflections: this means that we obtained a triangle which presents a 3-periodic trajectory  of the ball in the elliptic billiard. Poncelet's famous theorem  \cite{Poncelet} states that the sides of these triangles are tangent to some smaller ellipse confocal to the initial one.

We prove the following fact which was observed experimentally by Dan Reznik \cite{Reznik}: 

\begin{theorem}\label{1}
For every elliptic billiard the set of incenters (the centers of the inscribed circles) of  its triangular orbits is an ellipse.
\end{theorem} 

The proof uses very classical ideas: complexify and projectivize, that is, replace  the Euclidean plane by the complex projective plane. This approach was used by Ph. Griffiths and J. Harris in \cite{GH} and, more recently, by R. Schwartz in \cite{Schwartz}.
The main tool in the proof is that of complex reflection: we consider an ellipse as a complex curve and define a complex law of reflection off a complex curve. The locus of the incenters will be also a complex algebraic (even rational) curve. We will prove that the latter curve is a conic in $\mathbb{CP}^2$. Its real part will be a bounded conic -- an ellipse. 

The reasons for developping complex methods for the solution of a problem in planimetry are twofold: first of all, such an approach shows that sometimes complexification paradoxically simplifies things. We think that complex methods could be a useful tool in obtaining many results of this kind. Ideologically, this work is related with the recent work by A. Glutsyuk where he studies complex reflection, see for example \cite{1} an the joint work with Yu. Kudryashov \cite{GK}. The second reason to develop the complex approach for this particular problem was the incompetence of the author to prove this fact with real tools besides computation. The reader is encouraged to find an alternative proof of the Theorem \ref{1}.  

Complex reflection law and its basic properties needed here are reviewed in Section \ref{2}. Section \ref{3} contains the proof of Theorem \ref{1}. In Section \ref{sec4} we discuss the position of the foci of an obtained ellipse.

\section{Complex reflection law}\label{2}
For our purposes it will be useful to pass from the Euclidean plane $\mathbb{R}^2$ to the complex projective plane $\mathbb{CP}^2$: the metric now is replaced, in local complex coordinates $(z,w)$, by a quadratic form $ds^2=dz^2+dw^2$. In the following we will be interested in the geometry of this new space $\mathbb{CP}^2$ with quadratic form $ds^2$. One could have replaced the initial Euclidean metric by a pseudo-euclidean one: the geometry of such a space is also interesting and somewhat similar to our case. The best references here will be \cite{T} and \cite{GR}.

\begin{definition}
The lines with directing vectors that have zero length are called \emph{isotropic}. All other lines are called \emph{non-isotropic}.
\end{definition}

Let us  fix a point $x \in \mathbb{CP}^2$ and define  complex symmetry with respect to a line passing through $x$ as a map acting on the space $\mathcal{L}_x$ of all lines passing through $x$. There are two isotropic lines $L_{x}^{v_1}$ and $L_{x}^{v_2}$ in $\mathcal{L}_x$ with directing vectors $v_1= (1,i)$ and  $v_2 = (1,-i)$.

\begin{definition}[Complex reflection law]\label{CRL}
For a point $x\in\mathbb{CP}^2$, the \emph{complex reflection (symmetry)} in a non-isotropic line $L_x \in \mathcal{L}_x$ is the mapping given by the same formula as in the case of standard real symmetry: it's a linear map that in the coordinates defined by the line $L_x$ and its orthogonal line $L_x^{\perp}$ has a diagonal matrix $\begin{pmatrix}
1&0\\
0&-1
\end{pmatrix}.
$

The image of any line $L$ under reflection in an isotropic line $L_{x}^{v_1}$ (or $L_{x}^{v_2}$) is defined as a limit of its images under reflections with respect to a sequence of non-isotropic lines converging to $L_{x}^{v_1}$ (or $L_{x}^{v_2}$). 

Moreover, the complex reflection in a curve is the reflection in its tangent line.
\end{definition}

\begin{theorem}[\cite{1}]\label{111}
\begin{itemize}
\item[a.] The complex symmetry with respect to any isotropic line $L$ at some point $x \in L$ is well defined for all non-isotropic lines (i.e. the latter limit of the images of a sequence of non-isotropic lines exists) and maps every non-isotropic line $X\ni x$ to $L$.
 
\item[b.] Under the reflection at the point $x$ with respect to some isotropic line $L \in \mathcal{L}_x$, the line $L$ itself may be mapped to any line passing through $x$ (i.e. the mapping in this case is multivalued). In particular, it can stay fixed.
\end{itemize}
\end{theorem}

The isotropic directions generated by the vectors $v_{1}$ and $v_2$ can be represented by the points $I_{1}=(1: i:0) \in \mathbb{CP}^2$
and $I_{2}=(1:- i:0) \in \mathbb{CP}^2$, respectively. We  choose an affine coordinate $z$ on the projective line $\mathbb{CP}^1=\mathbb{C} \cup \infty$ at infinity, that is, the line through points $I_1$ and $I_2$ in such a way that $I_1=0$ and $I_2 =\infty$. The lemma below implies Theorem \ref{111} and follows easily from the definition. It describes the reflection in a line close to isotropic.

\begin{lemma}[\cite{1}]\label{11}
For any $\varepsilon \in \bar{\mathbb{C}} \setminus \left\{0, \infty\right\}$, let  $L_{\varepsilon}$ be the line through the origin $(0,0)\in \mathbb{C}^2$ and having direction $\varepsilon$. Let $\tau_{\varepsilon}: \mathbb{CP}^1 \rightarrow \mathbb{CP}^1$ be the reflection in  $L_{\varepsilon}$ acting on the space $\mathbb{CP}^1$ of the lines through the origin. Then $
\tau_{\varepsilon}(z)=\frac{\varepsilon^2}{z}$ in the above introduced coordinate $z$.
\end{lemma}

\begin{proof}
The map $\tau_{\varepsilon}$ is a projective transformation that  preserves $L_{\varepsilon}$ as well as the set of isotropic lines. So $\tau_{\varepsilon}(\varepsilon)=\varepsilon$ and $\tau_{\varepsilon}\{0, \infty\}=\{0, \infty\}$. Let us show that $\tau_{\varepsilon}$ permutes $0$ and $\infty$. Otherwise, it would have three fixed points on the infinity line $\mathbb{CP}^2 \setminus \mathbb{C}^2$ and therefore be identity map of the infinity line. Moreover, the points lying on $L_{\varepsilon}$ are fixed for $\tau_{\varepsilon}$. In this case $\tau_{\varepsilon}$ should be identity but it's a nontrivial involution, contradiction. 

We see that the restriction of $\tau_{\varepsilon}$ is a nontrivial conformal involution of $\mathbb{CP^2} \setminus \mathbb{C}$ fixing $\varepsilon$ and permuting $0$ and $\infty$. So it should map $z$ to $\frac{\varepsilon}{z^2}$. 
\end{proof}

\section{The proof}\label{3}
Let us consider triangular orbits of the complexified elliptic billiard: the triangles inscribed into a complexified ellipse and satisfying the complex reflection law. Denote the initial ellipse from Theorem \ref{1} by $\Gamma$, and the Poncelet ellipse tangent to all triangular orbits by $\gamma$. We use the same symbols for  complexifications of these conics.
  The following classical fact will be used for $\Gamma$ and $\gamma$, and for the inscribed circles.

\begin{lemma}[\cite{3}, p. $179$, \cite{Berger}, p.334]\label{12}
\begin{itemize}
\item[a.]
Ellipses $\Gamma$ and $\gamma$ in the real plane are confocal if and only if their complexifications have $4$ common isotropic tangent lines and their common foci lie on the intersections of these lines.
\item[b.]The two tangent lines to the complexified circle passing through its center are isotropic.
\end{itemize}
\end{lemma}

\begin{definition}[Sides and degenerate sides of a triangle]
A \emph{side} of a triangle in $\mathbb{CP}^2$ with disctinct vertices is a complex line through a pair of its vertices.  A triangle is called \emph{degenerate} if all its vertices lie on the same line. A priori, a triangular orbit may have coinciding vertices. We will call $A$ \emph{the degenerate side} through two coinciding vertices if $A$ is obtained as a limit of sides $A_{\varepsilon}, \varepsilon \rightarrow 0$ of non-degenerate triangular orbits. For such a side $A$ \emph{its image} under reflection is defined as a limit (which exists as the limit in Definition \ref{CRL}) of images of $A_{\varepsilon}$. 
\end{definition}

By taking a family $A_{\varepsilon}$ of lines tangent to $\gamma$ and converging to $A$, and computing their images (in fact, applying Lemma \ref{5} below), one could deduce the structure of degenerate triangular orbits formulated in Lemma \ref{lemma}. 

\begin{lemma}\label{5}
Let $A$ be a common isotropic tangent line to two analytic (algebraic) curves $\gamma$ and $\Gamma$ and let the tangency points be quadratic and distinct. If $A$ is deformed in a family $A_{\varepsilon}$ ($A=A_0$) of lines tangent to $\gamma$ then the image of $A_{\varepsilon}$ under the reflection in $\Gamma$ tends to some non-isotropic line as $\varepsilon \rightarrow 0$. 
\end{lemma}

\begin{proof}
The more general case of this lemma is contained in \cite{1}, see Proposition $2.8$ and Addendum $2$ there. 

The isotropic line $A$ is deformed in a family $A_{\varepsilon}$: let us suppose that the family is chosen in such a way that the angle between $A$ and $A_{\varepsilon}$ is precisely $\varepsilon$. Suppose that $A_{\varepsilon}$ intersects $\Gamma$ in some point $a_{\varepsilon}$ tending to the point $a_0$ of isotropic tangency. The simple computation shows that since the tangency points are quadratic, the tangent line $T_{\varepsilon}$ to $\Gamma$ at the point $a_{\varepsilon}$ has the angle of the order $\sqrt{\varepsilon}$ with $A$. This with lemma \ref{11} gives that the limit of the reflected lines is a non-isotropic one.
\end{proof}

Now we can describe the degenerate triangles occuring in our problem.
\begin{lemma}\label{lemma}
If a triangular orbit in the complexified ellipse $\Gamma$ is degenerate then it has two  coinciding non-isotropic non-degenerate sides $B$ and one degenerate isotropic side $A$. 
\end{lemma}

\begin{proof}
Since $\deg \Gamma = 2$, two vertices should merge, so the degenerate side $A$ through them is tangent to $\Gamma$ and to $\gamma$, and hence  is isotropic by Lemma \ref{12}. The other two sides are non-isotropic by Lemma \ref{5} and they coincide.
\end{proof}
\begin{lemma}[Main lemma]\label{main}
The complex curve of incenters $\mathcal{C}$ intersects the complex line $F$ through the foci of $\Gamma$ at exactly two points with index~$1$. 
\end{lemma}   

\begin{proof}
Let $c \in \mathcal{C} \cap F$ and suppose that the corresponding triangle is degenerate, see Figure \ref{picdegenerate}. By Lemma \ref{lemma} one of its sides is isotropic, and two other sides coincide and are non-isotropic. We will denote the isotropic line as $A$ and non--isotropic line as $B$. Line $A$ is tangent to the inscribed circle, so by Lemma \ref{12}, $c \in F \cap A$. Also $c$ is a point of intersection of  bisectors, so either $c \in B$ or $c \in B^{\perp}$. Note that $B$ is tangent to the inscribed circle, hence if $c \in B$, then $B$ should be isotropic, which is a contradiction.
So $c \in B^{\perp}$, but by Lemma \ref{12} $c$ is a focus. $B^{\perp}$ is tangent to $\Gamma$ and passes through the focus, so it should be isotropic which is impossible since $B$ is not isotropic.

\vskip4mm
\begin{figure}
\begin{center}
\includegraphics*[height=6.1cm]{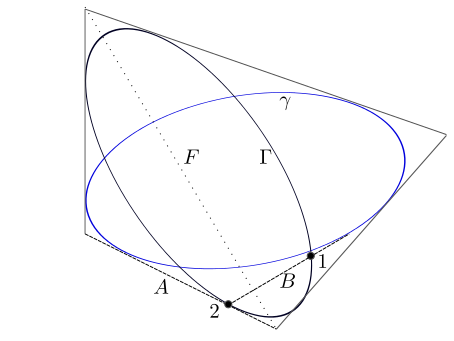}
\caption{Two complex confocal ellipses $\Gamma$ and $\gamma$ having four common isotropic tangent lines. The line $F$ of real foci passes through the intersections of isotropic lines. A degenerate trajectory for an elliptic billiard in $\Gamma$ with caustic $\gamma$: the degenerate triangle is an interval between points $1$ and $2$ and its sides are lines $A$ and $B$. Line $A$ is isotropic while $B$ is not.}\label{picdegenerate}
\end{center}
\end{figure}

Now consider the case of a not degenerate triangle corresponding to $c \in \mathcal{C} \cap F$. Consider the reflection in $F$: the inscribed circle, as well as its center $c$, are mapped to themselves. If the set of the sides of a triangle and their images under the reflection in $F$ consists of \emph{six} lines, then the inscribed circle and the ellipse $\gamma$ should be tangent to all of them. But  five tangent lines already define a conic, so $\gamma$ must be a circle. But in this case, Theorem \ref{1} is trivial and the locus under consideration is a point.

Therefore some sides of the triangle should map to some other sides. One needs to consider two cases: either there is a side which maps to itself, or there are two sides which map to each other. But the latter case  reduces to the former: indeed, the points of intersection of the two exchanging lines with $\Gamma$ (not lying on $F$) are mapped to each other, so the line connecting them is mapped to itself. Therefore, in the non-degenerate case, the corresponding triangle has a side which is symmetric with respect to $F$ and tangent to $\gamma$. There are only two such lines, and  hence only two intersections $c_1$ and $c_2$, both real (see Figure \ref{pic}), and only two triangles corresponding to them, for each $c_i, i=1,2$.

\vskip4mm
\begin{figure}
\begin{center}
\includegraphics*[height=6.1cm]{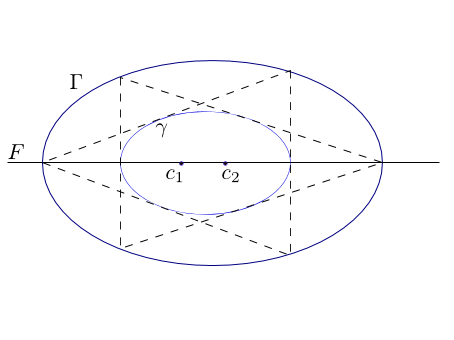}
\caption{Two triangular orbits in $\Gamma$ corresponding to the centers $c_1, c_2$ of inscribed circles lying on the foci line $F$}\label{pic}
\end{center}
\end{figure}

Let us now prove that the intersections $\mathcal{C} \cap F$ have index $1$. 
Let us parametrize an ellipse $\gamma$ by a parameter $\varepsilon$,  and consider the  corresponding center $c(\varepsilon) \in \mathcal{C}$, assuming that $c(0)\in F$. It suffices to prove that $\frac{\partial c}{\partial \varepsilon}(0) \neq 0$. Suppose the contrary: the centers of the circles do not change in the linear approximation: $c(\varepsilon) = c(0)+O(\varepsilon^2)$.
Then the  radius of the incircle $r(\varepsilon)$ has nonzero derivative at $\varepsilon=0$, unless for $\varepsilon=0$ both the incircle and the ellipse $\gamma$ are tangent to the sides of the triangle at the same points. This is impossible if $\gamma$ is not a circle, since two distinct conics can not be tangent at more than two distinct points. So we have that the radii of the incircles change linearly:  $r(\varepsilon) = r(0)+ \alpha \varepsilon(1+o(1))$ for $\alpha \neq 0$. But this is not possible due to  symmetry: indeed, the radius has to be an even function of $\varepsilon$.
\end{proof}

Theorem \ref{1} follows directly from the Lemma \ref{main} since an algebraic curve intersecting some line in exactly two points (with multiplicities) is a conic.

\section{Foci study}\label{sec4}
One could surmise that the ellipse $\mathcal{C}$ that is obtained in Theorem \ref{1} is confocal to the initial one. It appears that it is not so. Picture \ref{focipic} shows how the foci of the ellipse $\mathcal{C}$ move regarding the foci of the ellipse $\Gamma$. 

We suppose that the ratio between the semi-axis of the initial ellipse $\Gamma$ is $t \in (0,1)$. The upper graph on Figure \ref{focipic} is a graph of the distance from the center of $\Gamma$ to its foci: just the arc of the circle $\{\left(t, \sqrt{1-t^2}\right), t \in (0,1)\}$. The lower graph is a graph of analogous (quite complicated) function for the ellipse $\mathcal{C}$. This graph was obtained by pure computation. The reader is encouraged to find the geometrical meaning for the position of the foci of $\mathcal{C}$.
\vskip4mm
\begin{figure}
\begin{center}
\includegraphics*[height=5.5cm]{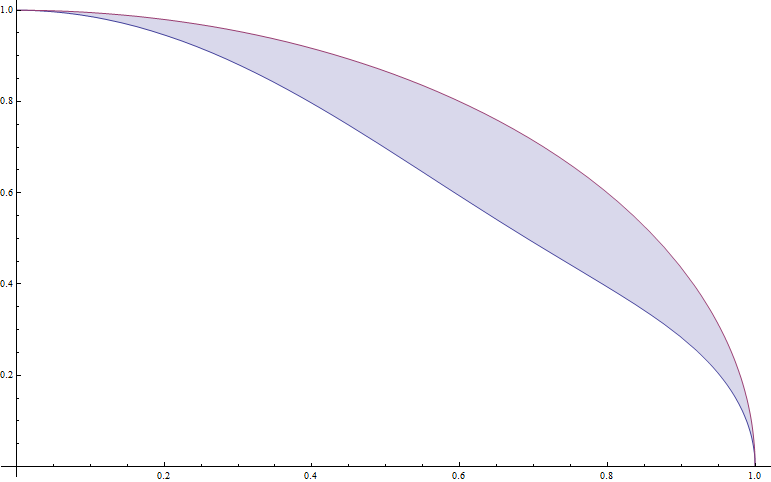}
\caption{Distances between the common center of ellipses $\Gamma$ and $C$ and their foci as functions of the ratio of semi-axes of an initial ellipse}\label{focipic}
\end{center}
\end{figure}

\section{Acknowledgments}
The author is grateful to Alexey Glutsyuk for stating the problem as well as for inspiring her with complex reflection ideas and for permanent encouragement. The author would also like to thank the referee for pointing out a considerable amount of references as well as making very helpful suggestions. The author is also grateful to A. Gorodentsev for providing the reference  \cite{3}, to I. Schurov for his succor in making pictures and to \'{E}. Ghys and S. Tabachnikov for support. This work was done in the pleasant atmosphere of \'Ecole Normale Sup\'erieure de Lyon.

\EMdate{24 april 2013}
%

 
\begin{address}
Olga Romaskevich \\
National Research University Higher School of Economics,
ENS de Lyon, UMPA, 
\email{oromaskevich@hse.ru, olga.romaskevich@ens-lyon.fr, \\
olga.romaskevich@gmail.com}
\end{address}
\end{document}